\documentclass[preprint]{elsarticle}
\usepackage{amssymb,amsmath,amsthm,latexsym}
\usepackage{psfrag,epsf}
\usepackage[left=2.75cm,right=2.75cm,top=3cm,bottom=3cm,a4paper]{geometry}
\newtheorem{theorem}{Theorem}[section]
\newtheorem{lemma}[theorem]{Lemma}

\theoremstyle{definition}

\numberwithin{equation}{section}

\journal{arXiv}

\begin{document}

\begin{frontmatter}
\title{Graphs having many holes but with small competition numbers}
\author[label1]{JungYeun Lee}
\author[label2]{Suh-Ryung Kim\fnref{label5}}
\author[label3]{Seog-Jin Kim}
\author[label4]{Yoshio Sano\corref{cor1}\fnref{label6}}
\address[label1]{National Institute for Mathematical Sciences,
Daejeon 305-390, Korea}
\address[label2]{Department of Mathematics Education,
Seoul National University,
Seoul 151-742, Korea}
\address[label3]{Department of Mathematics Education,
Konkuk University,
Seoul 143-701, Korea}
\address[label4]{
Pohang Mathematics Institute, POSTECH, Pohang 790-784, Korea}
%% \author[label1,label2]{}
%% \address[label1]{}
%% \address[label2]{}
\fntext[label5]{The author was supported 
by Basic Science Research Program through 
the National Research Foundation of Korea(NRF) funded by the Ministry
of Education, Science and Technology (700-20100058).}
\fntext[label6]{This work was supported
by Priority Research Centers Program
through the National Research Foundation of Korea (NRF) funded by
the Ministry of Education, Science and Technology (2009-0094069).}
 \cortext[cor1]{Corresponding author: ysano@postech.ac.kr}
%e-mail: srkim@snu.ac.kr
%\medskip

%e-mail: skim12@konkuk.ac.kr
%\medskip
%\date{}
%\maketitle

\begin{abstract}
The {\em competition number} $k(G)$ of a graph $G$ is the smallest
number $k$ such that $G$ together with $k$ isolated vertices added
is the competition graph of an acyclic digraph. A chordless cycle
of length at least $4$ of a graph is called a hole of the graph.
The number of holes of a graph is closely related to its
competition number as the competition number of a graph
which does not contain a hole is at most one and the competition
number of a complete bipartite graph 
$K_{\lfloor \frac{n}{2} \rfloor, \lceil \frac{n}{2} \rceil}$ 
which has so many holes that no more holes can be
added is the largest among those of graphs with $n$ vertices. In
this paper, we show that even if a connected graph $G$ has many
holes, 
the competition number of $G$ 
can be as small as $2$ under some assumption. 
In addition, we show that, 
for a connected graph $G$ with exactly $h$ holes 
and at most one non-edge maximal clique, if all the holes of
$G$ are pairwise edge-disjoint and the
clique number $\omega=\omega(G)$ of
$G$ satisfies $2 \leq \omega \leq h+1$, then
the competition number of $G$ is at most $h-\omega+3$.
\end{abstract}


\end{frontmatter}

%%%%%%%%%%%%%%%%%%%%%%%%%%%%%%%%%%%%%%%%%%%%%%%%%%%%%%%%%%%%%%%%%%%%%%%
\section{Introduction}
%%%%%%%%%%%%%%%%%%%%%%%%%%%%%%%%%%%%%%%%%%%%%%%%%%%%%%%%%%%%%%%%%%%%%%%

Let $D = (V,A)$ be a digraph (for all undefined
graph-theoretical terms, see \cite{bo}). The
\emph{competition graph} $C(D)$ of $D$ has the same vertex set as
$D$ and has an edge $xy$ if for some vertex $v \in V$, the arcs
$(x,v)$ and $(y,v)$ are in $D$. The notion of competition graph is
due to Cohen~\cite{cohen1} and has arisen from ecology.
A \emph{food web} in an ecosystem is a digraph whose vertices are
the species of the system and which has an arc from a vertex $u$
to a vertex $v$ if and only if $u$ preys on $v$. Given a food web
$F$, it is said that species $u$ and $v$ compete if and only if
they have a common prey. Competition graphs also have applications
in coding, radio transmission, and modeling of complex economic
systems.  (See \cite{RayRob} and \cite{Bolyai} for a summary of
these applications and \cite{glm} for a sample paper on the
modeling application.)

Roberts~\cite{Rob78} observed that every graph together with
sufficiently many isolated vertices is the competition graph of
an acyclic digraph.  The {\em competition number} $k(G)$ of a
graph $G$ is defined to be the smallest number $k$ such that $G$
together with $k$ isolated vertices added is the competition graph
of an acyclic digraph.  That is, when $I_k$ is a set of $k$
isolated vertices, $k(G)$ is the smallest integer $k$ such that
the disjoint union $G \cup I_k$ is the competition graph of an
acyclic digraph.  It is well known that computing the competition
number of a graph is an NP-hard problem \cite{Opsut}. It has been
one of the important research problems in the study of competition
graphs to characterize a graph by its competition number.

We call a cycle of a graph $G$ a \emph{chordless cycle} of $G$ if
it is an induced subgraph of $G$. A chordless cycle of length at
least 4 of a graph is called a \emph{hole} of the graph and a
graph without holes is called a {\em chordal graph}. The number of
holes of a graph is closely related to its competition number.
The competition number of a chordal graph
is at most one (see \cite{Rob78}).
The competition number
of a complete bipartite graph 
$K_{\lfloor \frac{n}{2} \rfloor,\lceil \frac{n}{2} \rceil}$ 
which has so many holes that no more holes can be added
is the largest among those of graphs with $n$ vertices
(see \cite{Harary}). 
Many authors have studied on the relationship between the number of holes
and the competition number of a graph
(see \cite{CK05}, \cite{Kim05}, \cite{KLS10}, \cite{twoholes}, \cite{LC}).
Roberts~\cite{Rob78} showed that if $G$ is nontrivial,
triangle-free and connected, then $k(G) = |E(G)| - |V(G)| + 2$.
In particular, if $G$ is a tree, then $k(G)=1$.
Take a graph $G$
such that $G$ has exactly $h$ holes
and no two holes of $G$ share an edge.
By the theorem by Roberts,
the competition number of $G$ is $h+1$
since $G$ has $h+|V(G)|-1$ edges.
Therefore $k(G)$ is almost as large as $h$.
Then we naturally come up with an interesting question:
``Is $k(G)$ still kept large if $G$ is allowed to have
just one maximal clique of size sufficiently large?".
In this paper, we answer this
question by showing that
even if a connected graph $G$ has many holes,
$k(G)$ can be as small as $2$ under some assumption.
In addition, we show that,
for a connected graph $G$ with exactly $h(G)$ holes
and at most one non-edge maximal clique,
if all the holes in $G$ are pairwise edge-disjoint and 
the clique number $\omega(G)$ of $G$
satisfies $2 \leq \omega(G) \leq h(G)+1$, 
then the competition number of $G$ is at most $h(G) -\omega(G)+3$.

%%%%%%%%%%%%%%%%%%%%%%%%%%%%%%%%%%%%%%%%%%%%%%%%%%%%%%%%%%%%%%%%%%%%%%%%%%%%
\section{Main Results}
%%%%%%%%%%%%%%%%%%%%%%%%%%%%%%%%%%%%%%%%%%%%%%%%%%%%%%%%%%%%%%%%%%%%%%%%%%%%

For a graph $G$ and a set $S \subseteq V(G)$ of vertices of $G$,
we denote by $G[S]$ the subgraph of $G$ induced by $S$.

\begin{lemma}\label{chord}
Let $C$ be a cycle of length at least $4$ in a graph $G$.
If $C$ has a chord, then the subgraph $G[V(C)]$ of $G$ has a triangle or
contains two holes which have a common edge.
\end{lemma}

\begin{proof}
Let $C=v_1v_2v_3 \ldots v_n$ be a cycle of $G$ and $v_iv_j$ be a chord
of $C$ for some $i<j$. Two $(v_i,v_j)$-sections of $C$ are
$(v_i,v_j)$-walks of $G[\{v_i,v_{i+1}, \ldots,v_j\}]-v_iv_j$ and
$G[\{v_j,v_{j+1}, \ldots,v_i\}]-v_iv_j$.
Let $P_1$ and $P_2$ be
shortest $(v_i,v_j)$-paths in
$G[\{v_i,v_{i+1}, \ldots,v_j\}]-v_iv_j$
and $G[\{v_j,v_{j+1}, \ldots,v_i\}]-v_iv_j$, respectively.
Since $G$ is simple,
the lengths of $P_1$ and $P_2$ are at least $2$.
If the length of $P_1$ or $P_2$ is $2$,
then $P_1+v_iv_j$ or $P_2+v_iv_j$
is a triangle in $G[V(C)]$.
Otherwise, $P_1+v_iv_j$ and $P_2+v_iv_j$ are holes
which have a common edge $v_iv_j$.
\end{proof}

A {\it clique} is a complete subgraph of a graph.
A clique $K$ is called {\it non-edge} if
$|V(K)| \geq 3$.
The {\it clique number} of a graph $G$ is the maximum number
of vertices of a clique in $G$
and is denoted by $\omega(G)$.

\begin{lemma}\label{chord2}
Let $G$ be a connected graph. Suppose that all the holes in $G$
are pairwise edge-disjoint and that $G$ has exactly one non-edge
maximal clique $K$. Then, a cycle $C$ in $G$ is a hole if and
only if it satisfies $|V(K) \cap V(C)| \leq 2$.
\end{lemma}

\begin{proof}
The `only if' part is obvious.
We show the `if' part by contradiction.
Suppose that $C$ is not a hole, that is, $C$ has a chord.
By Lemma~\ref{chord}, the subgraph $G[V(C)]$ of $G$ has a
triangle or contains two holes with a common edge.
If $G[V(C)]$ has a triangle,
then the triangle is a non-edge clique
different from $K$ since $|V(K) \cap V(C)| \leq 2$,
which is a contradiction.
Otherwise, it contradicts the assumption that all
the holes of $G$ are edge-disjoint.
Thus $C$ is a hole.
\end{proof}

For a clique $K$ in a graph $G$, we call a path $P$ in $G$ a
{\em $K$-avoiding path} if $P$ is not an edge of $K$ and any of
internal vertices of $P$ is not on $K$.

\begin{lemma}\label{cond}
Let $G$ be a connected graph with exactly $h$ holes.
Suppose that all the holes in $G$ are pairwise edge-disjoint
and that $G$ has exactly one non-edge maximal clique.
If the non-edge maximal clique $K$ in $G$ has size $h+1$,
then there exists a vertex $v$ in $K $ satisfying
one of the following:
\begin{itemize}
\item[{\rm (a)}]
there is no $K$-avoiding path from the vertex $v$
to any vertex in any hole,
\item[{\rm (b)}]
the vertex $v$ is incident to an edge common to $K$ and a hole,
and is not contained in any other hole.
\end{itemize}
\end{lemma}

\begin{proof}
Let
$H_1,H_2, \ldots,H_{h}$ be the holes of $G$.
We define a bipartite multigraph $B$ on bipartition $(V_1,V_2)$,
where $V_1 = V(K)=\{v_1,v_2, \ldots,v_{h+1}\}$ and
$V_2=\{H_1,H_2, \ldots,H_{h}\}$, as follows.
Two vertices $v_i \in V_1$ and $H_j \in V_2$
are joined by $r$ edges in $B$ if there
exists a $K$-avoiding path from $v_i$ to a vertex in $H_j$,
where $r$ is defined by
\[
r=
\left
\{\begin{array}{cl}
2 & \mbox{if $v_i$ is a cut vertex of $G$
and any vertex in $V(K) \setminus \{v_i\}$
and any vertex in $V(H_j) \setminus \{ v_i \}$ }\\
& \mbox{belong to different components of $G-v_i$,} \\
1 & \mbox{otherwise.}
\end{array}
\right.
\]
If $\deg_B(v_i)=0$ for some $i$,
then $v_i$ satisfies condition (a).
Suppose that $\deg_B(v_i)=1$ for some $i$.
Then there exists a unique $j$ such that
$G$ has a $K$-avoiding path $P$ from $v_i$
to a vertex $x$ in $H_j$.
Therefore $v_i$ is not contained in any
hole other than $H_j$.
If $G$ has no $K$-avoiding path from
$v_{i'} \in V(K) \setminus \{v_i\}$ to a vertex $x'$ in $H_j$,
then $v_i$ is a cut vertex and any vertex in
$V(K) \setminus \{v_i\}$ and any vertex in
$V(H_j) \setminus \{ v_i \}$ belong to
different components of $G-v_i$.
This implies that $\deg_B(v_i)=2$
and it is a contradiction.
Thus, $G$ has a $K$-avoiding path $P'$
from $v_{i'} \in V(K) \setminus \{v_i\}$ to a vertex $x'$ in
$H_j$. Then the walk formed by $v_i v_{i'}$, $P$, a
$(x,x')$-section of $H_j$, and $P'$ contains a cycle. Then the
edge $v_i v_{i'}$ is contained in a hole since $G$ has exactly one
non-edge maximal clique $K$.
Thus $v_i$ satisfies the condition (b).
Hence what we have to prove is the following:
\begin{itemize}
\item[($*$)]
there exists $v_i \in V_1$ such that $\deg_B(v_i) \leq 1$.
\end{itemize}
To show the claim ($*$), we show that $\deg_B(H_j) \leq 2$ hold
for all $1 \leq j \leq h$.
Suppose that $\deg_B(H_j) \geq 3$ for
some $j \in \{1, \ldots, h\}$.
We will reach a contradiction.

First, we suppose that there are three distinct $K$-avoiding paths
$P_1$, $P_2$, and $P_3$ going from the distinct vertices
$v_{i_1}$, $v_{i_2}$, and $v_{i_3}$ in $K$ to vertices $x_1$,
$x_2$, and $x_3$ in $H_j$, respectively.
Since $V(H_j) \cap V(K) \leq 2$ by Lemma~\ref{chord2},
without loss of generality, we may assume $v_{i_3} \notin V(H_j)$.
Then the length of $P_3$ is at least $1$.
Let $w$ be the vertex immediately following $v_{i_3}$ on $P_3$.
Then $w \not\in V(K)$.
If $v_{i_3}w$ is a cut edge of $G$,
then any path from a vertex in $K$ to a vertex in $H_j$ must
contain the edge $v_{i_3}w$.
This implies that $P_1$ contains the vertex $v_{i_3}$
as an internal vertex of $P_1$,
which contradicts that $P_1$ is a $K$-avoiding path.
Therefore $v_{i_3}w$ is not a cut edge,
and so the edge $v_{i_3}w$ is contained in some cycle in $G$.
Let $C$ be a shortest cycle among the cycles containing the
edge $v_{i_3}w$.
By the choice of $C$,  $C$ has no chord.
If $C$ is a triangle, i.e., a clique of size $3$,
then $C$ is a clique different from $K$
since $w \notin V(K)$ and $w \in V(C)$, which is a contradiction.
Thus $C$ is a hole.
Since $\{v_{i_1}, v_{i_2}, v_{i_3}\} \nsubseteq V(C)$
and $v_{i_3} \in V(C)$, $v_{i_1} \not\in V(C)$ or $v_{i_2} \not\in V(C)$.
Without loss of generality, we may assume that
$v_{i_1} \notin V(C)$.
The $(w,x_3)$-section of $P_3$, an $(x_3,x_1)$-section of $H_j$ and
the $(x_1,v_{i_1})$-section of $P_1$ form a $(w,v_{i_1})$-walk $W$
which does not contain $v_{i_3}$.
Let $Q$ be the shortest $(w,v_{i_1})$-path that is
a subsequence of the $(w,v_{i_1})$-walk $W$.
Then $C'=Qv_{i_3}w$ is a cycle.
Here we note that $ V(K) \cap V(C') = \{v_{i_1},v_{i_3}\}$
by the definition.
By Lemma~\ref{chord2}, $C'$ is a hole and we have reached a contradiction
as $v_{i_3}w$ is an edge common to the holes $C$ and $C'$.

Now suppose that $H_j \in V_2$ is incident to multiple edges. Let
$v_{i_1} \in V_1$ be the other end of the multiple edges. Since
$\deg_B(H_j) \geq 3$, there is another vertex $v_{i_2}$ adjacent
to $H_j$ in $B$. By the definition of $B$, $v_{i_1}$ is a cut
vertex of $G$ and no other vertex in $K$ belongs to the component
containing vertices of $H_j$ in $G-v_{i_1}$. It contradicts to the
existence of a $K$-avoiding path from $v_{i_2}$ to a vertex in
$H_j$ which does not contain $v_{i_1}$.

Consequently, $\deg_B(H_j) \leq 2$ for all $1 \leq j \leq h$
and so
\[
\sum^{h+1}_{i=1} \deg_B(v_i) = |E(B)|
= \sum^{h}_{j=1} \deg_B(H_j) \leq 2 h.
\]
If $\deg_B(v_i)\geq 2$ for all
$1 \leq i \leq h+1$, then $\sum^{h+1}_{i=1} \deg_B(v_i) \geq 2 (h+1)$
and it is a contradiction. Therefore, there exists a vertex $v_i$ with
$\deg_B(v_i) \leq 1$ and so ($\ast$) holds.
\end{proof}

\begin{lemma}\label{hole}
Let $G$ be a connected graph with exactly $h$ holes.
Suppose that all the holes in $G$ are pairwise edge-disjoint
and that $G$ has exactly one non-edge maximal clique $K$.
If $G-e$ has at least $h$ holes for some edge $e$ of a hole $H$ in $G$,
then $e$ is an edge of $K$.
In particular, holes in $G-e$ but not in $G$
have the form $(H - v_iv_j) \cup \{ v_iv_k, v_jv_k \}$
where $e=v_iv_j$ and $v_k$ is a vertex of $K$.
\end{lemma}

\begin{proof}
Suppose that $G-e$ has at least $h$ holes for an edge $e=uv$ of a
hole $H$. Since all the holes in $G$ are edge-disjoint, any hole
other than $H$  does not contain the edge $e$. Since $G-e$ has at
least $h$ hole, $e$ is a chord of a cycle distinct from $H$ in
$G$. That is, there exists a $(u,v)$-path $P$ other than $H-e$.
Without loss of generality, we may assume that $P$ is a shortest
path between $u$ and $v$ in $G-e$. Since $G$ is simple, $P$ is not
an edge. If the length of $P$ is at least $3$, then $P+e$ is a
hole which is distinct from $H$. It is also a contradiction as $e$
is an edge common to $H$ and $P+e$. Thus, the length of $P$ is
$2$. This implies that $P+e$ is a triangle and so it is contained
in $K$. Therefore, $e$ is an edge common to $H$ and $K$. In
addition, we can easily check that $H-e$ together edges $v_iv_k$
and $v_jv_k$ is a hole of $G-e$ where $e=v_iv_j$ and $v_k$ is a
vertex of $K$.
\end{proof}

\begin{lemma}\label{pasting}
Let $D_1$ and $D_2$ be acyclic digraphs such that
$V(D_1) \cap V(D_2) = \emptyset$.
Suppose that there are $p$ isolated vertices in $C(D_1)$ and
there are $p$ vertices which have no in-neighbors in $D_2$.
Then there exists an acyclic digraph $D$ such that
$C(D) = C(D_1) \cup C(D_2) - I_p$,
where $I_p$ is a set of $p$ isolated vertices in $C(D_1)$.
\end{lemma}

\begin{proof}
Let $I_p=\{ i_1$, $i_2$, \ldots, $i_p$\} be a set of $p$ isolated
vertices in $C(D_1)$ and $u_1$, $u_2$, \ldots, $u_p$ be vertices
which have no in-neighbors in $D_2$. We define a digraph $D$ with
vertex set $V(D_1) \cup V(D_2)- I_p$ by changing the arcs incoming
toward $i_j$ to the arcs incoming toward $u_j$, that is,
\[
A(D) = A(D_1) \cup A(D_2) -
\bigcup_{j=1}^p \{(v,i_j) \mid v \in N_{D_1}^{-}(i_j)\}
\cup \bigcup_{j=1}^p \{(v,u_j) \mid v \in N_{D_1}^{-}(i_j)\}
\]
(see Figure~\ref{paste} for an illustration).
Then $D$ is acyclic and $C(D) = C(D_1) \cup C(D_2) - I_p$.
Hence the lemma holds.
\end{proof}

\begin{figure}
\psfrag{D}{\footnotesize $D_1$}
\psfrag{H}{\footnotesize $D_2$}
\psfrag{G}{\footnotesize $C(D_1)$}
\psfrag{I}{\footnotesize $C(D_2)$}
\psfrag{J}{\footnotesize $D$}
\psfrag{K}{\footnotesize $C(D)=C(D_1) \cup C(D_2) - \{i\}$}
\psfrag{i}{\footnotesize $i$}
\psfrag{X}{\footnotesize $u$}
%\hskip0.5in
\begin{center}
\includegraphics[width=400pt]{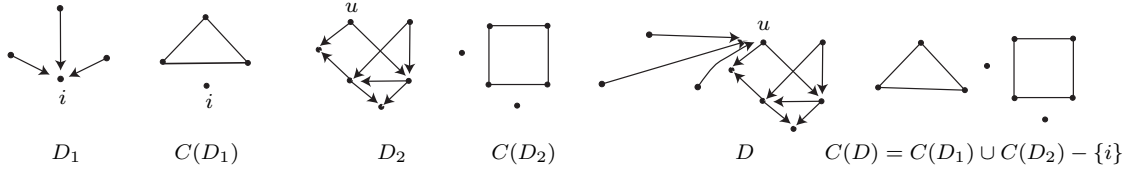}
\end{center}
\caption{\label{paste}
$D_1$, $D_2$, and $D$.}
\end{figure}

Now we show our main results. 
The following theorem claims that even 
if a connected graph $G$ has many holes, 
its competition number 
$k(G)$ can be as small as $2$.

\begin{theorem}\label{main1}
Let $G$ be a connected graph with exactly $h$ holes. 
Suppose that all the holes in $G$ are pairwise edge-disjoint 
and that $G$ has at most one non-edge maximal clique 
and that the clique number $\omega(G)$ of $G$ is equal to $h+1$. 
Let $K$ be a clique with $|V(K)|=\omega(G)$. 
Then, there exists an acyclic digraph $D$ 
such that $C(D)=G \cup \{i_1,i_2\}$ 
and all the vertices of $K$ have $i_2$ as a common out-neighbor, 
where $i_1$ and $i_2$ are new isolated vertices. 
In particular, $k(G) \leq 2$.
\end{theorem}

\begin{proof}
We prove by induction on the number of holes. 
Let $G$ be a connected graph with exactly one hole $H$. 
By the assumption that $\omega(G)=h+1(=2)$, 
the graph $G$ is triangle-free. 
Let $e=xy$ be an edge of $G$. 
We show that $k(G-e) \leq 1$. 
First, we consider the case $e \in E(H)$.
Then $G-e$ has no cycle by Lemma~\ref{hole}
since $G$ is triangle-free.
Therefore $G-e$ is a chordal graph and so $k(G-e) \leq 1$.
Second, we suppose that $e$ is a pendant edge of $G$.
Without loss of generality, 
we may assume that $x$ is a pendant vertex.
Then $G - e = G_1 \cup \{x\}$, 
where $G_1$ is a connected triangle-free graph
with exactly one hole.
Thus $|E(G_1)|=|V(G_1)|$ and so
$k(G_1) = |E(G_1)| - |V(G_1)| + 2 = 2$.
This implies that there exists an acyclic digraph $D_1$ 
such that $C(D_1)=G_1 \cup \{x, z_1\}$. 
Since $C(D_1)=(G_1 \cup \{x\}) \cup \{z_1\}= (G-e) \cup \{z_1\}$,
we have $k(G-e) \leq 1$. 
Finally, we deal with the case where $e$ is neither a 
pendant edge nor in $H$.
Then $e$ is a cut edge of $G$ since $e$ is
not in the unique cycle $H$ 
and so $G-e=G_1 \cup G_2$ 
where $G_1$ and $G_2$ are the connected components of $G-e$. 
Since $e$ is not a pendant edge, both $G_1$ and $G_2$ 
have at least two vertices. 
In addition, since the hole $H$ remains in $G-e$, 
either $G_1$ or $G_2$ is a tree. 
Without loss of generality,
we may assume that $H$ is in $G_1$ 
and $G_2$ is a tree. 
Then $k(G_2) \leq 1$, 
and so there exists an acyclic digraph $D_2$ 
such that $C(D_2)=G_2 \cup \{z_1\}$ 
where $i_1$ is an isolated vertex. 
Since we may take $D_2$ as a minimal acyclic digraph, 
$D_2$ contains two vertices $u$ and $v$ 
which have no in-neighbors in $D_2$. 
Since $G_1$ is connected, triangle-free and has exactly one hole,
$k(G_1)=|E(G_1)| - |V(G_1)| + 2 = 2$. 
Then there exists an acyclic 
digraph $D_1$ such that $C(D_1)=G_1 \cup \{z_2, z_3 \}$
where $z_2$ and $z_3$ are isolated vertices.
By Lemma~\ref{pasting}, 
there exists an acyclic digraph $D$
such that 
$C(D) = C(D_1) \cup C(D_2) - \{z_2, z_3\} = (G-e) \cup \{z_1\}$. 
Thus $k(G-e) \leq 1$. 
Hence, in any cases, we have $k(G-e) \leq 1$. 
Let $D'$ be an acyclic digraph such that 
$C(D')=(G-e) \cup \{i_1\}$, 
where $i_1$ is a new vertex. 
We define a digraph $D$ by
$V(D)=V(D') \cup \{i_2\}$ and
$A(D)=A(D') \cup \{(x, i_2), (y, i_2)\}$,
where $i_2$ is a new vertex. 
Then $D$ is acyclic and $C(D)=G \cup \{i_1, i_2\}$. 
Furthermore, both endpoints of $e$ have 
$i_2$ as a common out-neighbor in $D$. 
Hence the basis step holds. 

Now, we assume that, for any connected graph $\hat{G}$ 
with exactly $h$ ($\geq 1$) holes 
such that all the holes in $\hat{G}$ are pairwise edge-disjoint 
and that $\hat{G}$ has at most one non-edge maximal clique 
and that $\omega(\hat{G})=h+1$, 
there exists an acyclic digraph $D$ 
such that $C(D) = \hat{G} \cup \{i_1, i_2\}$ 
and all the vertices of the maximal clique have $i_2$
as a common out-neighbor in $D$. 
Note that if a graph has no non-edge maximal clique 
then the graph must have exactly one hole, 
which is already done in the above argument. 
So it is enough to consider only graphs which have exactly one non-edge 
maximal clique. 
Let $G$ be a connected graph 
with exactly $h+1$ holes 
such that all the holes in $G$ are pairwise edge-disjoint 
and that $G$ has exactly one non-edge maximal clique $K$ 
and that $\omega(G)=h+2$. 
Then $|V(K)|=h+2$. 
We denote the vertices of $K$ by $v_1, v_2, \ldots, v_{h+2}$ 
and the holes of $G$ by $H_1, H_2, \ldots, H_{h+1}$. 
By Lemma \ref{cond}, 
$K$ contains a vertex $v_i$ satisfying the condition (a) or (b). 
With out loss of generality, we may assume $v_i=v_1$. 

First, suppose that $v_1$ satisfies the condition (a). 
By Lemma \ref{hole}, 
$G-e$ has at most $h$ pairwise edge-disjoint holes 
for an edge $e=uw \in E(H_i) \setminus E(K)$. 
Consider the graph 
$G':= (G - e) - \{v_1v_j \mid v_j \in V(K) \setminus\{v_1\} \}$. 
Since $v_1$ satisfies (a), 
$v_1$ must belong to a component not 
containing holes or $u$ or $w$ in $G'$ 
and $G'$ has exactly two connected components. 
Let $G_1$ be the component containing  $v_1$ and 
$G_2$ be the other component of $G'$. 
Since $G_1$ is a tree and the competition number of a tree is equal to $1$, 
there exists an acyclic digraph $D_1$ 
such that $C(D_1)=G_1 \cup \{i_1\}$, where 
$i_1$ is a new isolated vertex, 
and that $D_1$ has at least two vertices, say $x$ and $y$, 
which have no in-neighbor in $D_1$. 
Since $G_2$ has a unique maximal clique, whose size is $h+1$, 
and exactly $h$ edge-disjoint holes, 
by the induction hypothesis, 
there exists an acyclic digraph $D_2$ 
such that $C(D_2) = G_2 \cup \{i_2,i_3\}$ 
where $i_2$ and $i_3$ are isolated vertices 
and all the vertices of $K - v_1$ 
have $i_2$ as a common out-neighbor in $D_2$. 
By Lemma~\ref{pasting},
there exists an acyclic digraph $D^*$ 
such that 
$C(D^*) = C(D_1) \cup C(D_2) - \{i_3\} 
= G_1 \cup G_2 \cup \{i_1, i_2\}$. 
Moreover, all the vertices of $K - v_1$ has a common 
out-neighbor $i_2$ in $D^*$. 
Now we add arcs 
$(v_1,i_2)$, $(u,y)$, $(w,y)$ to $D^*$
to obtain a digraph $D$. 
It can easily be checked that $D$ is acyclic and 
$C(D)=G \cup \{i_1,i_2\}$, 
and that all the vertices in $K$ have a common 
out-neighbor $i_2$. 

Second, we suppose that $v_1$ satisfies the condition (b). 
Then $v_1$ is incident to an edge $e$ shared by $K$ and a hole $H_j$, 
and $v_1$ is not a vertex on any other hole. 
Without loss of generality, we may assume $H_j=H_1$. 
Then 
$G':=G - \{ v_1v_j \mid v_j \in V(K) \setminus \{v_1\} \}$ 
has a unique maximal clique $K - v_1$. 
By Lemma~\ref{hole}, 
$G'$ has at most $h$ holes, which are pairwise edge-disjoint, 
since we removed all the edges incident to $v_1$ in $K$. 
By the induction hypothesis, 
there exists an acyclic digraph $D'$ 
such that $C(D')=G' \cup \{i_1,i_2\}$ 
where $i_1$ and $i_2$ are isolated vertices added 
and all the vertices of $K - v_1$ 
have a common out-neighbor $i_2$ in $D'$. 
Now, we define a digraph $D$ by 
$V(D)=V(G) \cup \{i_1,i_2\}$ and $A(D)=A(D') \cup \{(v_1,i_2)\}$.
Then it can easily be checked that $D$ is acyclic 
and $C(D)=G \cup \{i_1,i_2\}$ 
and that all the vertices in $K$ have a common out-neighbor $i_2$. 
Hence the theorem holds. 
\end{proof}

Theorem~\ref{main1} can be generalized as follows:

\begin{theorem}\label{main2}
Let $G$ be a connected graph. 
Suppose that all the holes in $G$ are pairwise edge-disjoint
and that $G$ has at most one non-edge maximal clique. 
If the clique number $\omega(G)$ of $G$ satisfies
$2 \leq \omega(G) \leq h(G)+1$ where 
$h(G)$ denotes the number of holes in $G$, 
then 
\begin{equation}\label{eq:ineq}
k(G) \leq h(G)-\omega(G) +3. \tag*{($\star$)}
\end{equation}
\end{theorem}

\begin{proof}
We prove by induction on the number $h(G)$ of holes in a graph $G$. 
Consider when 
$h(G)=1$. 
By $2 \leq \omega(G) \leq h(G)+1$, we have $\omega(G)=2$. 
It was shown in \cite[Theorem 11]{CK05} that 
$k(G) \leq 2$ holds for any graph $G$ with $h(G)=1$. 
Therefore $k(G) \leq 2 = h(G) - \omega(G) +3$, 
and thus the basis step holds. 
Now, we assume that the inequality \ref{eq:ineq} holds for 
any connected graph $\hat{G}$ with $h(\hat{G})=h$ ($h \ge 1$) 
such that all the holes in $\hat{G}$ are pairwise edge-disjoint
and that $\hat{G}$ has at most one non-edge maximal clique 
and that $2 \leq \omega(\hat{G}) \leq h(\hat{G})+1$.
Let $G$ be a connected graph with $h(G)=h+1$ 
such that all the holes in $G$ are pairwise edge-disjoint 
and that $G$ has at most one non-edge maximal clique 
and that $2 \leq \omega(G) \leq h(G)+1(=h+2)$. 
If $\omega(G) = 2$, then $h(G) -\omega(G) +3 = h+2$. 
It was shown in \cite[Theorem 1.5]{KLS10} 
that $k(G) \leq h(G)+1$ holds for any graph $G$ such that 
all the holes in $G$ are pairwise edge-disjoint. 
Therefore the inequality \ref{eq:ineq} holds. 
If $\omega(G) = h+2$, then $h(G) - \omega(G) + 3 = 2$. 
By Theorem~\ref{main1}, $k(G) \le 2$ 
and the inequality \ref{eq:ineq} holds. 
Thus we assume that $3 \leq \omega:=\omega(G) \leq h+1$. 
Let $e=xy$ be an edge on some hole $H$ 
but not on the non-edge maximal clique $K$ in $G$. 
Then $G-e$ is a connected graph with $h(G-e)=h$ 
and $\omega(G-e) = \omega$ 
such that all the holes of $G-e$ are pairwise edge-disjoint 
and that $G$ has at most one non-edge maximal clique. 
Since $\omega < h+2$, we have $\omega(G-e) \leq h(G-e)+1$. 
By the induction hypothesis, 
$k(G-e) \leq h(G-e)-\omega(G-e) +3= h - \omega+3$. 
Then there exists an acyclic digraph $D'$ 
such that $C(D')=(G-e) \cup I_{h - \omega + 3}$. 
Now we let $D$ be the digraph obtained 
by adding a new vertex $i$ and two new arcs $(x,i)$ and $(y,i)$ to $D'$. 
Then the digraph $D$ is acyclic and 
$C(D)=G \cup I_{h - \omega + 3} \cup \{i\}$. 
Thus $k(G) \leq (h+1) - \omega+3$. 
Hence the theorem holds. 
\end{proof}

%%%%%%%%%%%%%%%%%%%%%%%%%%%%%%%%%%%%%%%%%%%%%%%%%%%%%%%%%%%%%%%%%%%%%%%%%%%
%{99}%


{\small
\begin{thebibliography}{99}%

\bibitem{bo}
{J. A. Bondy and U. S. R. Murty}:
{\it Graph Theory with Applications},
(North Holland, New York, 1976).

\bibitem{CK05}
{H. H. Cho and S.-R. Kim}:
{The competition number of a graph having exactly one hole},
{\it Discrete Math.} {\bf 303} (2005) 32--41.

\bibitem {cohen1}
{J. E. Cohen}:
{Interval graphs and food webs: a finding and a problem},
{Document 17696-PR}, RAND Corporation,
Santa Monica, CA (1968).

\bibitem{glm}
{H. J. Greenberg, J. R. Lundgren, and J. S. Maybee}:
Graph-theoretic foundations fo computer-assisted analysis,
in H. J. Greenberg, J. S. Maybee (eds.),
{Computer-Assisted Analsys and Model Simplification},
Academic Press, New York, 1981, 481--495.

\bibitem {Harary}
{F. Harary, S.-R. Kim, and F. S. Roberts}:
Extremal competition numbers as a generalization of Turan's theorem,
{\it J. Ramanujan Math. Soc.} {\bf 5} (1990) 33--43.

%\bibitem {Kim93}
%{S.-R. Kim}:
%{The competition number and its variants},
%{ Quo Vadis, Graph Theory}, (J. Gimbel, J. W. Kennedy, and L. V.
%Quintas, eds.), {\em Annals of Discrete Mathematics} {\bf 55},
%North-Holland, Amsterdam (1993) 313--326.

\bibitem{Kim05}
{S.-R. Kim}:
{Graphs with one hole and competition number one},
{\it J. Korean Math. Soc.} {\bf 42} (2005) 1251--1264.

\bibitem{KLS10}
{S.-R. Kim, J. Y. Lee, and Y. Sano}:
{The competition number of a graph whose holes do not overlap much},
{\it Discrete Appl. Math.} {\bf 158} (2010) 1456--1460.

\bibitem{twoholes}
{J. Y. Lee, S.-R. Kim, S.-J. Kim, and Y. Sano}:
{The competition number of a graph with exactly two holes},
{\it Ars Combin.} {\bf 95} (2010) 45--54.

\bibitem{LC}
{B.-J. Li and G. J. Chang}:
{The competition number of a graph
with exactly $h$ holes, all of which are independent},
{\it Discrete Appl. Math.} {\bf 157} (2009) 1337--1341.

\bibitem{Opsut}
{R. J. Opsut}:
{On the computation of the competition number of a graph},
{\it SIAM J. Algebraic Discrete Methods} {\bf 3} (1982) 420--428.

\bibitem{RayRob}
{A. Raychaudhuri and F. S. Roberts}:
{Generalized competition graphs and their applications},
in P. Br\"{u}cker and R. Pauly (eds.),
{\it Methods of Operations Research},
{\bf 49} Anton Hain, K\"{o}nigstein, West Germany, (1985) 295--311.

\bibitem{Rob78}
{F. S. Roberts}:
{Food webs, competition graphs, and the boxicity of ecological phase space},
in Y. Alavi and D. Lick (eds.),
{\it Theory and Applications of Graphs (Proc. Internat. Conf.,
Western Mich. Univ., Kalamazoo, Mich., 1976)} (1978) 477--490.

\bibitem {Bolyai}
{F. S. Roberts}:
Competition graphs and phylogeny graphs, in L.
Lovasz (ed.), {\it Graph Theory and Combinatorial Biology},
{\it Bolyai Mathematical Studies}, Vol. {7}, {\it J. Bolyai
Mathematical Society}, Budapest (1999) 333--362.

\end{thebibliography}}
\end{document}